\documentclass{amsart}

\usepackage{color,graphicx,amssymb,latexsym,amsfonts,txfonts,amsmath,amsthm}
\usepackage{pdfsync}
\usepackage{amsmath,amscd}
\usepackage[all,cmtip]{xy}

\usepackage{hyperref}
\hypersetup{
    colorlinks=true,       
    linkcolor=blue,          
    citecolor=blue,        
    filecolor=blue,      
    urlcolor=blue           
}

\input epsf

\newtheorem{theo}{Theorem}
\newtheorem{coro}{Corollary}
\newtheorem{prop}{Proposition}
\newtheorem{lemm}{Lemma}

\theoremstyle{remark}
\newtheorem{rema}{\bf Remark}

\newtheorem{example}{\bf Example}

\begin{document}

\title{On quasiconformal equivalence of Schottky regions }

\author{Rub\'en A. Hidalgo}
\address{Departamento de Matem\'atica y Estad\'{\i}stica, Universidad de La Frontera. Temuco, Chile}
\email{ruben.hidalgo@ufrontera.cl}

\thanks{Partially supported by Project Fondecyt 1230001}



\begin{abstract}
In a recent paper, H. Shiga proved that the regions of discontinuity of any two Schottky groups of ranks at least two are quasiconformally equivalent. In this paper, we provide an alternative proof of such a fact. Our approach permits us to discuss the quasiconformal equivalence of regions of discontinuity of Schottky type groups in terms of their signatures.

\end{abstract}

\maketitle

\section{Introduction}
Two (connected) Riemann surfaces are said to be quasiconformally equivalent if there is a quasiconformal homeomorphism between them. Associated to each Riemann surface $S$ is its 
Teichm\"uller space $T(S)$, which consists of equivalence classes of quasiconformal homeomorphisms $\phi:S \to R$, where $R$ is a Riemann surface (two of them $\phi:S \to R$ and $\psi:S \to R'$ are equivalent if there is biholomorphism $\eta:R \to R'$ such that $\psi^{-1} \circ  \eta \circ   \phi$ is isotopic to the identity map). In this way, given two homeomorphic Riemann surfaces, one is interested in determining if they are or not quasiconformally equivalent. It is well known that two topologically finite Riemann surfaces are quasiconformally equivalent if and only if they have the same genus, the same number of punctures, and the same number of boundary components. On the other hand, given two topologically infinite Riemann surfaces, it is a difficult task to determine if they are or not quasiconformally equivalent.

By the classification in \cite{Ker,Ian}, every planar infinite type surface is homeomorphic to the complement of a closed subset of a Cantor set in the sphere. In this way, interesting examples of infinite type Riemann surfaces are provided by deleting a Cantor set from the Riemann sphere $\widehat{\mathbb C}$.
In \cite{Shiga,Shiga2}, H. Shiga studies the above quasiconformal problem for these complements of Cantor sets. 

Examples of such kind of planar Riemann surfaces are given by the region of discontinuity of a finitely Kleinian group (i.e., a discrete subgroup of M\"obius transformations) whose limit set is totally disconnected. Some Kleinian groups with such a property are the so-called {\it Schottky type (ST) groups of signature} $(\alpha,\beta,\gamma)$, where $\alpha, \beta, \gamma$ are non-negative integers such that $\alpha+\beta+\gamma \geq 2$ (see Section \ref{Sec:STgroups}). 
 If $\beta=\gamma=0$, then they are also called {\it Schottky groups of rank} $\alpha$. Two Schottky type groups are quasiconformally conjugated if and only if they have the same signature.
As a consequence of Maskit's classification of function groups \cite{Maskit:function2,Maskit:function3,Maskit:function4}, every finitely generated Kleinian group whose limit set is totally disconnected contains, as a finite index subgroup, an ST group (so they have the same region of discontinuity). In this way, to study when the regions of discontinuity of two finitely generated Kleinian groups with totally disconnected limit sets are or not quasiconformally equivalent, one only needs to restrict to the case of ST groups.

In \cite{Shiga}, H. Shiga proved that the regions of discontinuity of any two Schottky groups of ranks at least two (they might have different ranks) are quasiconformally equivalent. In the same paper, he observed that, for $\alpha \geq 2$, the regions of discontinuity of ST groups of signatures $(\alpha-1,1,0)$  and $(\alpha,0,0)$ cannot be quasiconformally equivalent.

In this paper, which might be considered a complement to Shiga's paper, we provide conditions in terms of the signatures for 
the regions of discontinuity of any two ST groups to be quasiconformally equivalent. Our method, applied to Schottky groups, permits to recover Shiga's quasiconformal rigidity theorem for Schottky groups. Two triples of non-negative integers $(\alpha_{1},\beta_{1},\gamma_{1})$ and $(\alpha_{2},\beta_{2},\gamma_{2})$ are called {\it quasiconformally equivalent}, denoted this by the symbol $(\alpha_{1},\beta_{1},\gamma_{1}) \cong_{QC} (\alpha_{2},\beta_{2},\gamma_{2})$, if the regions of discontinuity of ST groups of these corresponding signatures are quasiconformally equivalent. The main result is the following.

\begin{theo}\label{Main}
Let $\Gamma$ be an ST group of signature $(\alpha,\beta,\gamma)$, where $\alpha+\beta+\gamma \geq 2$. Then
\begin{enumerate}
\item $(\alpha,0,0) \cong_{QC} (2,0,0)$.
\item If $\beta \geq 1$, then $(\alpha,\beta,0) \cong_{QC} (0,2,0)$.
\item If $\gamma \geq 1$, then $(\alpha,0,\gamma) \cong_{QC} (0,0,2)$.
\item The signatures $(2,0,0)$, $(0,2,0)$ and $(0,0,2)$ are pairwise quasiconformally non-equivalent.
\item If $\beta,\gamma \geq 1$, then $(\alpha,\beta,\gamma) \cong_{QC} (0,1,1)$.
\end{enumerate}
\end{theo}

\begin{coro}
There are exactly four non-equivalent Teichm\"uller spaces of Cantor limit sets of Kleinian groups.
\end{coro}

The main idea in the proof of Theorem \ref{Main} is a combination of the following observations. First, any two ST groups of the same signature are quasiconformally conjugated, in particular, their regions of discontinuity are quasiconformally equivalent. If two ST groups are finite index subgroups of a common ST group, then the three of them have the same region of discontinuity. If two ST groups of different signatures contain, as a finite index, ST groups of the same signature, then their regions of discontinuity are quasiconformally equivalent. In Section \ref{Sec:Construcciones}, we construct finite index subgroups which are used with the above observations.

\section{Preliminaries}
In this section, we recall some definitions and properties of Kleinian groups we will need. Generalities on Kleinian groups can be found, for instance, in the classical book \cite{Maskit:book}.

\subsection{Kleinian groups}
A Kleinian group is a discrete subgroup $\Gamma$ of ${\rm PSL}_{2}({\mathbb C})$ (seen as the group of conformal automorphisms of the Riemann sphere $\widehat{\mathbb C}$). Its region of discontinuity $\Omega_{\Gamma}$ is the open set (which it might be empty) consisting of the points of $\widehat{\mathbb C}$ over which $\Gamma$ acts discontinuously. The complement $\Lambda_{\Gamma}=\widehat{\mathbb C} \setminus \Omega_{\Gamma}$ is called its limit set (if it is not finite, then $\Gamma$ is called non-elementary). 
If $K$ is a finite index subgroup of a Kleinian group $\Gamma$, then $\Omega_{K}=\Omega_{\Gamma}$ \cite[Proposition E.10]{Maskit:book}. The same result is still valid if we assume $\Gamma$ to be non-elementary and $K$ to be a non-trivial normal subgroup.

If $\Gamma$ is a finitely generated Kleinian group for which there is a $\Gamma$-invariant connected component of $\Omega_{\Gamma}$, then it is called a function group.
 A description of function groups, in terms of the Klein-Maskit combination theorems  \cite{Maskit:Comb,Maskit:Comb4}, was provided by Maskit in \cite{Maskit:construction, Maskit:function2, Maskit:function3, Maskit:function4}.

\subsection{Schottky type groups}\label{Sec:STgroups}
Let $\alpha, \beta, \gamma \geq 0$ be integers. A Kleinian group $\Gamma$ is called a {\it Schottky type group} (or ST group) of signature $(\alpha,\beta,\gamma)$ if it can be constructed as a free product (in the sense of the Klein-Maskit combination theorem \cite{Maskit:Comb,Maskit:Comb4}) of:
\begin{enumerate}
\item  $\alpha$ cyclic groups generated by loxodromic transformations, say $\langle A_{1} \rangle,$
$ \ldots, \langle A_{\alpha} \rangle$ (if $\alpha \geq 1$);

\item $\beta$ rank one parabolic groups (cyclic groups generated by a parabolic transformation), say $\langle B_{1} \rangle, \ldots, \langle B_{\beta} \rangle$ (if $\beta \geq 1$), and 

\item $\Gamma$ rank two parabolic groups (i.e, each one generated by a pair of commuting parabolic transformations), say $\langle C_{1},D_{1}\rangle, \ldots, \langle C_{\gamma}, D_{\gamma} \rangle$ (if $\gamma \geq 1$). 
\end{enumerate}

We say that the above generators are {\it geometrical generators} of $\Gamma$.

The following properties hold.
\begin{enumerate}
\item If $\alpha=\beta=\gamma=0$, then $\Gamma=\{I\}$ and $\Omega_{\Gamma}=\widehat{\mathbb C}$.
\item If $\alpha+\beta+\gamma=1$, then $\#(\Lambda_{\Gamma}) \in \{1,2\}$ (equal to $1$ exactly for $\alpha=0$).
\item If $\alpha+\beta+\gamma \geq 2$, then $\Lambda_{\Gamma}$ is a Cantor set.
\item $\Omega_{\Gamma}$ is connected and $\Omega_{\Gamma}/\Gamma$ is an analytically finite Riemann surface of genus $\alpha+\gamma$ with exactly $2\beta$ punctures.
\end{enumerate}

\begin{rema}\label{Obs1}
In \cite{Chuckrow}, these groups are called groups of {\it extended Schottky type} $(\alpha,\beta,\gamma)$. A Schottky type group of signature $(\alpha,0,0)$ is a 
Schottky group of rank $\alpha$. 
\end{rema}

A consequence of the geometric description of ST groups by the Klein-Maskit combination theorem is the following fact.

\begin{theo}\label{teo1}
Let $\Gamma_{j}$ be a ST group of signature $(\alpha_{j},\beta_{j},\gamma_{j})$, where $j \in \{1,2\}$. Then there is a quasiconformal homeomorphism $f:\widehat{\mathbb C} \to \widehat{\mathbb C}$ such that $f \Gamma_{1} f^{-1}=\Gamma_{2}$ if and only if $(\alpha_{1},\beta_{1},\gamma_{1})=(\alpha_{2},\beta_{2},\gamma_{2})$. In particular, in such a situation, $f$ provides a quasiconformal homeomorphism between $\Omega_{\Gamma_{1}}$ and $\Omega_{\Gamma_{2}}$.
\end{theo}

As a consequence of the Klein-Maskit combination theorem, the above ST group $\Gamma$ is torsion-free (i.e., it only contains parabolic and loxodromic transformations besides the identity). Moreover, every parabolic transformation is conjugated in $\Gamma$ to a power of either one of the parabolic elements $B_{j}, C_{t}, D_{t}$. 
Also, the limit set $\Lambda_{\Gamma}$ is a totally disconnected set (a Cantor set), so an ST group is a function group. 

Conversely, Maskit's classification of function groups asserts that a torsion-free function group, whose limit set is totally disconnected, is obtained as a free product of cyclic groups generated by parabolic and/or loxodromic transformations and groups generated by two commuting parabolic transformations. This asserts the following characterization of ST groups.

\begin{prop}\label{Prop1}
A torsion-free function group is an ST group if and only if its limit set is totally disconnected.
\end{prop}

Another characterization was provided by Chuckrow in \cite[Proposition 2]{Chuckrow}).

\begin{prop}[\cite{Chuckrow}]\label{Prop2}
A Kleinian group is an ST group if and only if its finitely generated, torsion-free and the only univalent analytic functions on its region of discontinuity are restrictions of M\"obius transformations. 
\end{prop}

The above permits us to note the following fact on subgroups of ST groups.

\begin{prop}\label{Prop3}
Every finitely generated subgroup of an ST group is an ST group. In particular, every finite index subgroup of an ST group is an ST group.
\end{prop}

\section{Computing some subgroups of ST groups}\label{Sec:Construcciones}
We denote by  $(\alpha_{1},\beta_{1},\gamma_{1}) \leq_{n} (\alpha_{2},\beta_{2},\gamma_{2})$ to indicate that a ST group of signature $(\alpha_{2},\beta_{2},\gamma_{2})$ contains a ST group of signature $(\alpha_{1},\beta_{1},\gamma_{1})$ as an index $n$ subgroup.
We denote by ${\mathfrak S}_{n}$ the symmetric group in $n$ letters.

By Proposition \ref{Prop3}, any finite index subgroup of a ST group $\Gamma$ of signature $(\hat{\alpha},\hat{\beta},\hat{\gamma})$, where $\hat{\alpha}+\hat{\beta}+\hat{\gamma} \geq 2$, is also a ST group of some signature. Every such subgroup, say of index $n \geq 2$, is obtained (up to conjugation in $\Gamma$) as $\Theta^{-1}({\rm Stab}_{\Theta(\Gamma)}(n))$, where 
$\Theta:\Gamma \to {\mathfrak S}_{n}$ is a homomorphism with $\Theta(\Gamma)$ a transitive subgroup. Below, we construct some of these homomorphisms $\Theta$, which will be used in the proof of Theorem \ref{Main} and, for each of them we describe the signature of the corresponding ST subgroup.

In our constructions, we make use of the following observation (see Figure \ref{Figura1}).

\begin{lemm}\label{lemita}
\begin{enumerate}
\item Let $G=\langle T,B \rangle$ be a ST group of signature either $(1,1,0)$ ($T$ is loxodromic and $B$ parabolic) or $(0,2,0)$ (both $T$ and $B$ parabolics). Then the subgroup $K=\langle TB, TB^{-1} \rangle$ is a ST group of signature $(1,1,0)$.

\item Let $G=\langle T,C,D \rangle$ be a ST group of signature either $(1,0,1)$ ($T$ is loxodromic and $C$, $D$ commuting parabolics) or $(0,1,1)$ ($T$, $C$ and $D$ parabolics). Then the subgroup $K=\langle TCT^{-1}, TD, TD^{-1} \rangle$ is a ST group of signature $(1,0,1)$.
\end{enumerate}
\end{lemm}
\begin{proof}
(1) Let us observe that $K=\langle U=TB, V=TB^{2}T^{-1}\rangle$, where $U$ is loxodromic and $V$ is parabolic (left side in Figure \ref{Figura1}).
(2) In this case, $K=\langle U=TD, V=TCT^{-1}, W=TD^{2}T^{-1}\rangle=\langle U \rangle * \langle V,W\rangle$, where $U$ is loxodromic and $V$ and $W$ are commuting parabolics (right side in Figure \ref{Figura1}).
\end{proof}

Let us fix a set of geometrical generators for $\Gamma$.
Let (i) $A_{1},\ldots,A_{\hat{\alpha}}$ loxodromic elements, (ii) $B_{1},\ldots,B_{\hat{\beta}}$ parabolic elements and (iii) $C_{1}, D_{1},\ldots,C_{\hat{\gamma}},D_{\hat{\gamma}}$ be pairs of parabolic elements such that $C_{j}$ and $D_{j}$ commute, such that all of them provide a set of geometrical generators (in the Klein-Maskit combination construction).

We set $\sigma=(1,2,\ldots,n)$ and $\tau_{r}=(1,2)(3,4)\cdots(2r-1,2r)$, where $1 \leq r \leq n/2$.

\subsection{\bf Construction A: }\label{ConstA} 
\mbox{}\newline
Let  $\hat{\alpha} \geq 1$, $0 \leq r_{1},\ldots,r_{\hat{\beta}} \leq n/2$ (if $\hat{\beta} \geq 1$) and $0 \leq s_{1},\ldots,s_{\hat{\gamma}} \leq n/2$ (if $\hat{\gamma} \geq 1$).

Let $\Theta$ be defined by the following rules:
$$\Theta(A_{1})=\sigma, \; \Theta(A_{2})=\cdots=\Theta(A_{\hat{\alpha}})=1,$$
$$\Theta(B_{j})=\tau_{r_{j}},\; j=1,\ldots,\hat{\beta},$$ 
$$\Theta(C_{t})=1,\; \Theta(D_{t})=\tau_{s_{t}}, \; t=1,\ldots,\hat{\gamma}.$$

In this case,  $\Theta^{-1}({\rm Stab}_{\Theta(\Gamma)}(n))$ is generated by
$$A_{1}^{n},$$
$$A_{1}^{l}A_{i}A_{1}^{-l}, \; l=0,\ldots,n-1,\; i=2,\ldots,\hat{\alpha},$$
$$A_{1}^{l}B_{j}A_{1}^{-l}, \; l=0,\ldots,n-2r_{j}-1,\; j=1,\ldots,\hat{\beta},$$
$$A_{1}^{l+1}B_{j}A_{1}^{-l},  A_{1}^{l+1}B_{j}^{-1}A_{1}^{-l}, \; l=n-2r_{j},n-2r_{j}+2,\ldots,n-2,\; j=1,\ldots,\hat{\beta},$$
$$A_{1}^{l}C_{t}A_{1}^{-l},\; l=0,\ldots,n-1,\; t=1,\ldots,\hat{\gamma},$$
$$A_{1}^{l}D_{t}A_{1}^{-l},\; l=0,\ldots,n-2s_{t}-1,\; t=1,\ldots,\hat{\gamma},$$
$$A_{1}^{l+1}D_{t}A_{1}^{-l},  A_{1}^{l+1}D_{t}^{-1}A_{1}^{-l},\; l=n-2s_{t},n-2s_{t}+2,\ldots,n-2,\; t=1,\ldots,\hat{\gamma}.$$

The above is an ST group of signature 
$$\left(1+n(\hat{\alpha}-1)+\sum_{j=1}^{\hat{\beta}}r_{j}+\sum_{t=1}^{\hat{\gamma}}s_{t},n\hat{\beta}-\sum_{j=1}^{\hat{\beta}}r_{j},n\hat{\gamma}-\sum_{t=1}^{\hat{\gamma}}s_{t}\right)$$
 which is an index $n$ subgroup of an ST group of signature $(\hat{\alpha},\hat{\beta},\hat{\gamma})$.

\begin{example}\label{Ej1}
Some particular examples of the above construction are the following ones.
\begin{enumerate}
\item If we set $r_{j}=s_{t}=0$, then we obtain $$(1+n(\hat{\alpha}-1),n \hat{\beta},n \hat{\gamma}) \leq_{n} (\hat{\alpha},\hat{\beta},\hat{\gamma}).$$

\item If $\hat{\alpha}=\hat{\beta}=1$, $\hat{\gamma}=0$ and $r_{1}=r$, then 
$$(1+r,n-r,0) \leq _{n} (1,1,0).$$

\item If $\hat{\alpha}=\hat{\gamma}=1$, $\hat{\beta}=0$, $s_{1}=r$, then 
$$(1+r,0,n-r) \leq _{n} (1,0,1).$$

\item If $\hat{\alpha}=\hat{\beta}=\hat{\gamma}=1$,  $s_{1}=r_{1}=r$, then 
$$(1+2r,n-r,n-r) \leq_{n} (1,1,1).$$

\item If $n=2$, $\sum_{j=1}^{\hat{\beta}}r_{j}=\hat{\beta}-r$, $\sum_{t+1}^{\hat{\gamma}}s_{t}=\hat{\gamma}-s$, then 
$$(2\hat{\alpha}+\hat{\beta}+\hat{\gamma}-r-s-1,\hat{\beta}+r,\hat{\gamma}+s) \leq_{2} (\hat{\alpha},\hat{\beta},\hat{\gamma}).$$

\item If $\hat{\beta}=\hat{\gamma}=\gamma$, $\hat{\alpha}=\alpha$, $\sum_{j=1}^{\gamma}r_{j}=\sum_{t=1}^{\gamma}s_{t}=r \in \{0,\ldots,\gamma n/2\}$, then 
$$(1+n(\alpha-1)+2r,n\gamma-r,n\gamma-r) \leq_{n} (\alpha,\gamma,\gamma).$$

\end{enumerate}
\end{example}

\subsection{\bf Construction B: }\label{ConstB} 
\mbox{}\newline
Let  $\hat{\beta} \geq 1$,  $0 \leq r_{2},\ldots,r_{\hat{\beta}} \leq n/2$ (if $\hat{\beta} \geq 2$) and $0 \leq s_{1},\ldots,s_{\hat{\gamma}} \leq n/2$ (if $\gamma \geq 1$).

Let $\Theta$ be defined by the following rules:
$$\Theta(A_{i})=1, \; i=1,\ldots,\hat{\alpha},$$
$$\Theta(B_{1})=\sigma, \; \Theta(B_{j})=\tau_{r_{j}},\; j=2,\ldots,\hat{\beta},$$
$$\Theta(C_{t})=1,\; \Theta(D_{t})=\tau_{s_{t}}, \; t=1,\ldots,\hat{\gamma}.$$

In this case,  $\Theta^{-1}({\rm Stab}_{\Theta(\Gamma)}(n))$ is generated by
$$B_{1}^{l}A_{i}B_{1}^{-l}, \; l=0,\ldots,n-1,\; i=1,\ldots,\hat{\alpha},$$
$$B_{1}^{l}B_{j}B_{1}^{-l}, \; l=0,\ldots,n-2r_{j}-1,\; j=2,\ldots,\hat{\beta},$$
$$B_{1}^{l+1}B_{j}B_{1}^{-l},  B_{1}^{l+1}B_{j}^{-1}B_{1}^{-l}, \; l=n-2r_{j},n-2r_{j}+2,\ldots,n-2,\; j=2,\ldots,\hat{\beta},$$
$$B_{1}^{l}C_{t}B_{1}^{-l},\; l=0,\ldots,n-1,\; t=1,\ldots,\hat{\gamma},$$
$$B_{1}^{l}D_{t}B_{1}^{-l},\; l=0,\ldots,n-2s_{t}-1,\; t=1,\ldots,\hat{\gamma},$$
$$B_{1}^{l+1}D_{t}B_{1}^{-l},  B_{1}^{l+1}D_{t}^{-1}B_{1}^{-l},\; l=n-2s_{t},n-2s_{t}+2,\ldots,n-2,\; t=1,\ldots,\hat{\gamma}.$$

The above is an ST group of signature 
$$\left(n\hat{\alpha}+\sum_{j=2}^{\hat{\beta}}r_{j}+\sum_{t=1}^{\hat{\gamma}}s_{t},1+(\hat{\beta}-1)n-\sum_{j=2}^{\hat{\beta}}r_{j},n\hat{\gamma}-\sum_{t=1}^{\hat{\gamma}}s_{t}\right)$$
 which is an index $n$ subgroup of an ST group of signature $(\hat{\alpha},\hat{\beta},\hat{\gamma})$.

\begin{example}\label{Ej2} 
Some particular examples of the above construction are the following ones.
\begin{enumerate}
\item If we set $r_{j}=s_{t}=0$, then we obtain $$(n\hat{\alpha},1+n(\hat{\beta}-1),n\hat{\gamma}) \leq_{n} (\hat{\alpha},\hat{\beta},\hat{\gamma}).$$

\item If $n=2$, $\hat{\alpha}=\hat{\gamma}=0$, $\hat{\beta}=2$, $r_{1}=1$ and $r_{2}=0$, then $$(1,2,0) \leq_{2} (0,2,0).$$

\item If $\hat{\alpha}=0$, $\hat{\beta}=\hat{\gamma}=1$, $r_{1}=0$ and $s_{1}=r$, then 
$$(r,1,n-r) \leq _{n} (0,1,1).$$

\item If $\hat{\alpha}=0,\hat{\beta}=2,\hat{\gamma}=1$, $r_{2}=r$, $s_{1}=s$,  then $$(r+s,n+1-r,n-s) \leq_{n} (0,2,1).$$

\item If $n=2$, $\hat{\alpha}=0$, $\hat{\beta} \geq 2$, $\sum_{j=1}^{\hat{\beta}}r_{j}=\hat{\beta}-1$, $\sum_{t=1}^{\hat{\gamma}}s_{t}=\hat{\gamma}$, then 
$$(\hat{\beta}+\hat{\gamma}-1,\hat{\beta},\hat{\gamma}) \leq_{2} (0,\hat{\beta},\hat{\gamma}).$$

\end{enumerate}
\end{example}

\begin{rema}\label{observa3}
If in Example \ref{Ej2}(3) we use $n=2$ and $r=1$, we obtain an ST group of signature $(1,1,1)$ as an index two subgroup of an ST group of signature $(0,1,1)$. 
\end{rema}

\subsection{\bf Construction C: }\label{ConstC}
\mbox{}\newline
Let  $\hat{\gamma} \geq 1$, $0 \leq r_{1},\ldots,r_{\hat{\beta}} \leq n/2$ (if $\hat{\beta} \geq 1$) and $0 \leq s_{2},\ldots,s_{\hat{\gamma}} \leq n/2$ (if $\hat{\gamma} \geq 2$).

Let $\Theta$ be defined by the following rules:
$$\Theta(A_{i})=1, \; i=1,\ldots,\hat{\alpha},$$
$$\Theta(B_{j})=\tau_{r_{j}}, \; j=1,\ldots,\hat{\beta},$$
$$\Theta(C_{t})=1,\; t=1,\ldots,\hat{\gamma},$$
$$\Theta(D_{1})=\sigma, \; \Theta(D_{t})=\tau_{s_{t}}, \; t=2,\ldots,\hat{\gamma}.$$

In this case,  $\Theta^{-1}({\rm Stab}_{\Theta(\Gamma)}(n))$ is generated by
$$D_{1}^{n},$$
$$D_{1}^{l}A_{i}D_{1}^{-l}, \; l=0,\ldots,n-1,\; i=1,\ldots,\hat{\alpha},$$
$$D_{1}^{l}B_{j}D_{1}^{-l}, \; l=0,\ldots,n-2r_{j}-1,\; j=1,\ldots,\hat{\beta},$$
$$D_{1}^{l+1}B_{j}D_{1}^{-l},  D_{1}^{l+1}B_{j}^{-1}D_{1}^{-l}, \; l=n-2r_{j},n-2r_{j}+2,\ldots,n-2,\; j=1,\ldots,\hat{\beta},$$
$$D_{1}^{l}C_{t}D_{1}^{-l},\; l=0,\ldots,n-1,\; t=1,\ldots,\hat{\gamma},$$
$$D_{1}^{l}D_{t}D_{1}^{-l},\; l=0,\ldots,n-2s_{t}-1,\; t=2,\ldots,\hat{\gamma},$$
$$D_{1}^{l+1}D_{t}D_{1}^{-l},  D_{1}^{l+1}D_{t}^{-1}D_{1}^{-l},\; l=n-2s_{t},n-2s_{t}+2,\ldots,n-2,\; t=2,\ldots,\hat{\gamma}.$$

The above is an ST group of signature 
$$\left(n\hat{\alpha}+\sum_{j=1}^{\hat{\beta}}r_{j}+\sum_{t=2}^{\hat{\gamma}}s_{t},n\hat{\beta}-\sum_{j=1}^{\hat{\beta}}r_{j},1+n(\hat{\gamma}-1)-\sum_{t=2}^{\hat{\gamma}}s_{t}\right)$$
 which is an index $n$ subgroup of an ST group of signature $(\hat{\alpha},\hat{\beta},\hat{\gamma})$.

\begin{example}\label{Ej3}
Some particular examples of the above construction are the following ones.
\begin{enumerate}
\item If we set $r_{j}=s_{t}=0$, then we obtain $$(n\hat{\alpha},n\hat{\beta},1+n(\hat{\gamma}-1)) \leq_{n} (\hat{\alpha},\hat{\beta},\hat{\gamma}).$$

\item If $n=2$, $\hat{\alpha}=\hat{\beta}=0$, $\hat{\gamma}=2$, $s_{1}=1$ and $s_{2}=0$, then $$(1,0,2) \leq_{2} (0,0,2).$$

\item If $n=2$, $\hat{\alpha}=0$, $\hat{\gamma} \geq 2$, $\sum_{j=1}^{\hat{\beta}}r_{j}=\hat{\beta}$, $\sum_{t=2}^{\hat{\gamma}}s_{t}=\hat{\gamma}-1$, then 
$$(\hat{\beta}+\hat{\gamma}-1,\hat{\beta},\hat{\gamma}) \leq_{2} (0,\hat{\beta},\hat{\gamma}).$$

\end{enumerate}
\end{example}

\section{Proof of Theorem \ref{Main}}
By the constructions in the previous section, we observe that there are non-elementary ST groups of different signatures whose regions of discontinuity are quasiconformally equivalent. But others are not quasiconformally equivalent. For example, in \cite{Shiga}, H. Shiga observed that the signatures $(\alpha,0,0)$ and $(\alpha-1,1,0)$, where $\alpha \geq 2$, cannot be quasiconformally equivalent. 

\begin{rema}\label{Observa4}
Following the same argument as in \cite{Shiga}, one may observe that the signatures $(2,0,0)$ and $(0,2,0)$ cannot be quasiconformally equivalent. 
Also, the signature $(0,0,2)$ cannot be quasiconformally equivalent to the either signature $(2,0,0)$ and $(0,2,0)$. This follows from the fact that (i) the limit set of an ST group of signature $(0,0,2)$ is not contained in a quasicircle and (ii) the limit sets of ST groups of signatures $(2,0,0)$ and $(0,2,0)$ are. This, in particular, provides part (4) of Theorem \ref{Main}.
\end{rema}

\begin{rema}\label{Const4}
Let $n,b \geq 1$ and $\beta \geq 1$ such that $2\beta+1 \leq nb$. If 
$0 \leq r_{j} \leq n/2$, for $j=1,\ldots,b$, and set $r=r_{1}+\cdots+r_{b} \in \{0,1,\ldots,bn/2\}$, then construction \ref{ConstA} asserts that
$(1,b,0) \leq_{n} (1+r,nb-r,0)$. Note that, by Example \ref{Ej1}(1), this signature is equivalent to $(1,1,0)$. If $\alpha=1+r$ and $\beta=nb-r$, then $2r \leq 1+nb$ asserts that $\alpha \geq \beta$. This kind of computations permits us to observe that every signature of the form $(\alpha,\beta,0)$, where $\alpha \geq \beta \geq 1$, is equivalent to one of the form $(\hat{\alpha},\hat{\beta},0)$, where $1 \leq \hat{\alpha} \leq \hat{\beta}$, and that they are all equivalent to $(1,1,0)$.

Similarly, every signature of the form $(\alpha,0,\gamma)$, where $\alpha \geq \gamma \geq 1$, is equivalent to one of the form $(\hat{\alpha},0,\hat{\gamma})$, where $1 \leq \hat{\alpha} \leq \hat{\gamma}$, and that they are all equivalewnt to $(1,0,1)$.
\end{rema}

Below, we observe two simple conditions to obtain quasiconformal equivalence between two triples, which will be useful in what follows.

\begin{lemm}\label{Lem1}
Let $j \in \{1,2\}$ and $\Gamma_{j}$ be a non-elementary ST group of signature $(\alpha_{j},\beta_{j},\gamma_{j})$.
\begin{enumerate}
\item If there is an ST group $\Gamma$  containing, as a finite index subgroup, two ST groups of respective signatures $(\alpha_{1},\beta_{1},\gamma_{1})$ and $(\alpha_{2},\beta_{2},\gamma_{2})$, then $(\alpha_{1},\beta_{1},\gamma_{1}) \cong_{QC} (\alpha_{2},\beta_{2},\gamma_{2})$.

\item Let us assume $\Gamma_{j}$ contains, as a finite index subgroup, an ST group $K_{j}$. If $K_{1}$ and $K_{2}$ have the same signature, then 
$(\alpha_{1},\beta_{1},\gamma_{1}) \cong_{QC} (\alpha_{2},\beta_{2},\gamma_{2})$.
\end{enumerate}
\end{lemm}
\begin{proof}
(1) If $\Gamma$ is a non-elementary ST group of signature $(\alpha,\beta,\gamma)$ and $H$ is a finite index subgroup, then both have the same region of discontinuity. This observation, together with Theorem \ref{teo1}, permits us to obtain the desired claim.
(2) As $K_{1}$ and $K_{2}$ have the same signature, there is a quasiconformal homeomorphism between $\Omega_{K_{1}}$ and $\Omega_{K_{2}}$. As $K_{j}$ has finite index in $\Gamma_{j}$, it follows that $\Omega_{K_{j}}=\Omega_{\Gamma_{j}}$.
\end{proof}

Next, we will divide the signatures into those with the property that $\beta\gamma=0$ and those with $\beta\gamma \neq 0$.

\subsection{Case $\beta\gamma=0$: Proof of parts (1), (2) and (3) of Theorem \ref{Main}}
Let us recall from Remark \ref{Observa4} that the three signatures $(2,0,0)$, $(0,2,0)$ and $(0,0,2)$ are pairwise quasiconformally non-equivalent. In the following, we observe that a signature $(\alpha,\beta,\gamma)$, where $\beta\gamma=0$, is quasiconformally equivalent to (exactly) one of them. This provides 
a generalization of  \cite[Theorem 6.2]{Shiga} (which, in the theorem below is case (1)). 

\begin{theo}\label{Coro1}
Let $(\alpha,\beta,\gamma)$, where $\alpha,\beta,\gamma \geq 0$, $\beta\gamma=0$ and $\alpha+\beta+\gamma \geq 2$.
Then
\begin{enumerate}
\item $(\alpha,\beta,\gamma) \cong_{QC} (2,0,0)$ if and only if $\beta=\gamma=0$.
\item $(\alpha,\beta,\gamma) \cong_{QC} (0,2,0)$ if and only if $\beta \geq 1$ and $\gamma=0$.
\item $(\alpha,\beta,\gamma) \cong_{QC} (0,0,2)$ if and only if $\beta=0$ and $\gamma \geq 1$. 
\end{enumerate}
\end{theo}
\begin{proof}
(1) Assume $\beta=\gamma=0$, so $\alpha \geq 2$. If, in Example \ref{Ej1}(1), we set $(\hat{\alpha},\hat{\beta},\hat{\gamma})=(2,0,0)$ and $n=\alpha-1$, then we obtain a Schottky group of rank $\alpha$ as an index $n$ subgroup of a Schottky group of rank two, so $(\alpha,0,0) \cong_{QC} (2,0,0)$. 

(2) Assume $\gamma=0$, $\beta \geq 1$ and $\alpha+\beta \geq 2$. If $\alpha=0$, then $\beta \geq 2$ and this follows from Example \ref{Ej2}(1) by taking $(\hat{\alpha},\hat{\beta},\hat{\gamma})=(0,2,0)$ and $n=\beta-1$. So, let us now assume $\alpha \geq 1$.
By Example \ref{Ej2}(2),  $(1,2,0) \cong_{QC} (0,2,0)$ and, by Example \ref{Ej1}(1), $(1,2,0) \cong_{QC} (1,1,0)$. So, by Lemma \ref{Lem1},  $(0,2,0) \cong_{QC} (1,1,0)$. Now, if $\alpha \leq \beta+1$, then if Example \ref{Ej1}(4) we take $r=\alpha-1$ and $n=\alpha+\beta-1$, then $(\alpha,\beta,0)=(1+r,n-r,0) \cong_{QC} (1,1,0)$. 
If $\beta \leq \alpha+1$, then it follows from Remark \ref{Observa4} and the previous case.

(3) Assume $\beta=0$ and $\gamma \geq 1$. If $\alpha=0$, so $\gamma \geq 2$, then it follows from Example \ref{Ej3}(1), by taking $(\hat{\alpha},\hat{\beta},\hat{\gamma})=(0,0,2)$ and $n =\gamma- 1$, that there is a ST group of signature $(0,0,\gamma)$ as an index $n$ subgroup of a Schottky type group of signature $(0,0,2)$. Let us now assume that $\alpha \geq 1$.
By Example \ref{Ej1}(1), $(1,0,2) \cong_{QC} (1,0,1)$ and, 
by Example \ref{Ej3}(2), $(1,0,2) \cong_{QC} (0,0,2)$. This asserts that $(0,0,2) \cong_{QC} (1,0,1)$.
Now, similarly as done in the previous case (2), but now applying Example \ref{Ej1}(3) (and Remark \ref{Observa4}), we may obtain $(\alpha,0,\gamma)\cong_{QC}(1,0,1)$.
\end{proof}

\subsection{Case $\beta\gamma \neq 0$: Proof of part (5) of Theorem \ref{Main}}
Now, let us consider a signature of type $(\alpha,\beta,\gamma)$, where $\beta,\gamma \geq 1$. 
The following observation permits us to assume, up to quasiconformal equivalence, that $\beta=\gamma$. 

\begin{prop}[Part (5) of Theorem \ref{Main}]\label{Lem3}
Let $\beta,\gamma \geq 1$ be integers and set $M={\rm Max}\{\beta,\gamma\}$. Then there is some $\tilde{\alpha} \geq 2\alpha+\beta+\gamma-2 \geq 1$ such that 
$(\alpha,\beta,\gamma) \cong_{QC} (\tilde{\alpha},M,M)$.
\end{prop}
\begin{proof}
By Examples \ref{Ej2}(5), up to quasiconformal equivalence, we may assume $\alpha \geq 1$. 

 Set $\alpha_{0}=\alpha$, $\beta_{0}=\beta$ and $\gamma_{0}=\gamma$. 
If $\beta=\gamma$, then we are done.
Let us assume $\beta<\gamma$.
Note that, by Example \ref{Ej1}(5) applied to a triple $(\alpha_{k},\beta_{k},\gamma_{k})$ (for each $k \geq 0$), with $r=1$ and $s=0$, produces a triple $(\alpha_{k+1},\beta_{k+1},\gamma_{k+1})$, where
$\alpha_{k+1}=2\alpha_{k}+\beta_{k}+\gamma_{k}-2$, $\beta_{k+1}=\beta_{k}+1$ and $\gamma_{k+1}=\gamma_{k}$. We note that 
$\gamma_{k+1}=\gamma$, $\beta_{k+1}=\beta+k+1$ and $\alpha_{k+1}=2\alpha_{k}+\beta+\gamma+k-1$.
So, if we take $k=\gamma-\beta-1$, then we obtain the triple $(\alpha_{\gamma},\gamma,\gamma)$ as desired.
If $\gamma<\beta$, then we proceed in the same way as before by taking $r=0$ and $s=1$ in Example \ref{Ej1}(5).
\end{proof}

As a consequence of Proposition \ref{Lem3}, we only need to consider those signatures of the form $(\alpha,\gamma,\gamma)$, where $\alpha,\gamma \geq 1$.

Let us recall that (see Remark \ref{observa3}) that $(1,1,1) \cong_{QC} (0,1,1)$ and, by construction \ref{ConstA} that $(1,\gamma,\gamma) \cong_{QC} (1,1,1)$, for $\gamma \geq 1$.

\begin{lemm}\label{Lem4}
Let $\hat{\gamma} \geq 1$.
\begin{enumerate}
\item If $0 \leq r \leq \hat{\gamma}/2$, then $(1+2r,\hat{\gamma}-r,\hat{\gamma}-r) \cong_{QC} (0,1,1)$.

\item If $0 \leq r \leq \hat{\gamma}$, then $(2\hat{\gamma}-2r+1,\hat{\gamma}+r,\hat{\gamma}+r) \cong_{QC} (0,1,1)$.
\end{enumerate}
\end{lemm}
\begin{proof}
Part (1) follows from Example \ref{Ej1}(4) and 
part (2) follows from Example \ref{Ej1}(5): if $0 \leq r \leq \hat{\gamma}$, then $(2\hat{\gamma}-2r+1,\hat{\gamma}+r,\hat{\gamma}+r) \cong_{QC} (1,\hat{\gamma},\hat{\gamma}) \cong_{QC} (1,1,1) \cong_{QC} (0,1,1)$. 
\end{proof}

\begin{prop}\label{Lem5}
If $\gamma \geq 1$ and $\alpha \leq 2\gamma+1$, then $(\alpha,\gamma,\gamma) \cong_{QC} (0,1,1)$.
\end{prop}
\begin{proof}
Let us first consider the case $\alpha \geq 1$.
By Example \ref{Ej1}(1), $(\alpha,\gamma,\gamma) \cong_{QC} (\tilde{\alpha},\tilde{\gamma},\tilde{\gamma})= (1+2(\alpha-1),2\gamma,2\gamma)$. Note that 
$\tilde{\alpha} \leq 2\tilde{\gamma}+1$. So, we may restrict to assume that $\alpha \geq 1$ is odd.
Set $\alpha=1+2r$ and $\hat{\gamma}=\gamma+r$. The condition that $\alpha \leq 2\gamma+1$ is equivalent to 
$r \leq \hat{\gamma}/2$. Then part (1) of Lemma \ref{Lem4} asserts that $(\alpha,\gamma,\gamma)=(1+2r,\hat{\gamma}-r,\hat{\gamma}-r) \cong_{QC} (0,1,1)$.
Now, if  $\alpha=0$, then, by Examples \ref{Ej2}(5) and \ref{Ej3}(3), $(0,\gamma,\gamma) \cong_{QC} (2\gamma-1,\gamma,\gamma) \cong_{QC} (0,1,1)$.
\end{proof}

The above proposition takes care of part (5) of our theorem in the case $\alpha \leq 2\gamma+1$. So, we now need to take care of those signatures $(\alpha,\gamma,\gamma)$ where $\alpha>2\gamma+1$. Moreover, by Examples \ref{Ej2}(5) and \ref{Ej3}(3), we only need to take care of the cases when $\alpha$ is odd. 

\begin{lemm}\label{Lem6}
If $\alpha \geq 2\gamma+1$ is odd, where $\gamma \geq 1$, then $(\alpha,\gamma,\gamma) \cong_{QC} (\frac{\alpha+1-2\gamma}{2},\gamma,\gamma)$.
\end{lemm}
\begin{proof}
Let $a=(\alpha+1-2\gamma)/2$ and consider the signature $(a,\gamma,\gamma)$. By Example \ref{Ej1}(5), applied to the signature $(a,\gamma,\gamma)$ and $r=s=0$, we obtain that
$(a,\gamma,\gamma) \cong_{QC} (\alpha,\gamma,\gamma)$.
\end{proof}

\begin{example}
Proposition \ref{Lem5} asserts that $(\alpha,1,1) \cong_{QC} (0,1,1)$ under the assumption that $\alpha \in \{0,1,2,3\}$. 
(i) Let us consider the signature $(4,1,1)$. 
By Example \ref{Ej1}(1), applied to $(4,1,1)$ and $n=2$, we have that $(4,1,1) \cong_{QC} (7,2,2)$. Now, Lemma \ref{Lem6} asserts that $(7,2,2) \cong_{QC} (2,2,2)$. By Proposition \ref{Lem5}, $(2,2,2) \cong_{QC} (0,1,1)$. So, $(4,1,1) \cong_{QC} (0,1,1)$.

(ii) Let us now consider the signature $(5,1,1)$. By Lemma \ref{Lem6}, $(5,1,1) \cong_{QC} (2,1,1)$. By Proposition \ref{Lem5}, $(2,1,1) \cong_{QC} (0,1,1)$.  So, $(5,1,1) \cong_{QC} (0,1,1)$.

\end{example}

The ideas in the above example can be generalized to other cases.

\begin{prop}\label{Lem7}
If $\alpha \geq 2\gamma+1$, $\gamma \geq 1$, then $(\alpha,\gamma,\gamma) \cong_{QC} (0,1,1)$.
\end{prop}
\begin{proof}
By Example \ref{Ej1}(1), applied to $(\alpha,\gamma,\gamma)$ and $n=2$, we have that $(\alpha,\gamma,\gamma) \cong_{QC} (1+2(\alpha-1),2\gamma,2\gamma))=(2\alpha-1,2\gamma,2\gamma)$. Now, by Lemma \ref{Lem6}, $(2\alpha-1,2\gamma,2\gamma)\cong_{QC} ((2\alpha-4\gamma)/2,2\gamma,2\gamma)=(\alpha-2\gamma,2\gamma,2\gamma)=(\alpha-\tilde{\gamma},\tilde{\gamma},\tilde{\gamma})$, for $\tilde{\gamma}=2\gamma$. Continuing with this process, applied to the last signature, we may obtain a signature $(\hat{\alpha},\hat{\gamma},\hat{\gamma})$ for which $\hat{\alpha} \leq 2\hat{\gamma}+1$.
By Proposition \ref{Lem5}, $(\hat{\alpha},\hat{\gamma},\hat{\gamma}) \cong_{QC} (0,1,1)$. 
\end{proof}


\begin{figure}[ht]
\centering
\includegraphics[width=0.3\linewidth]{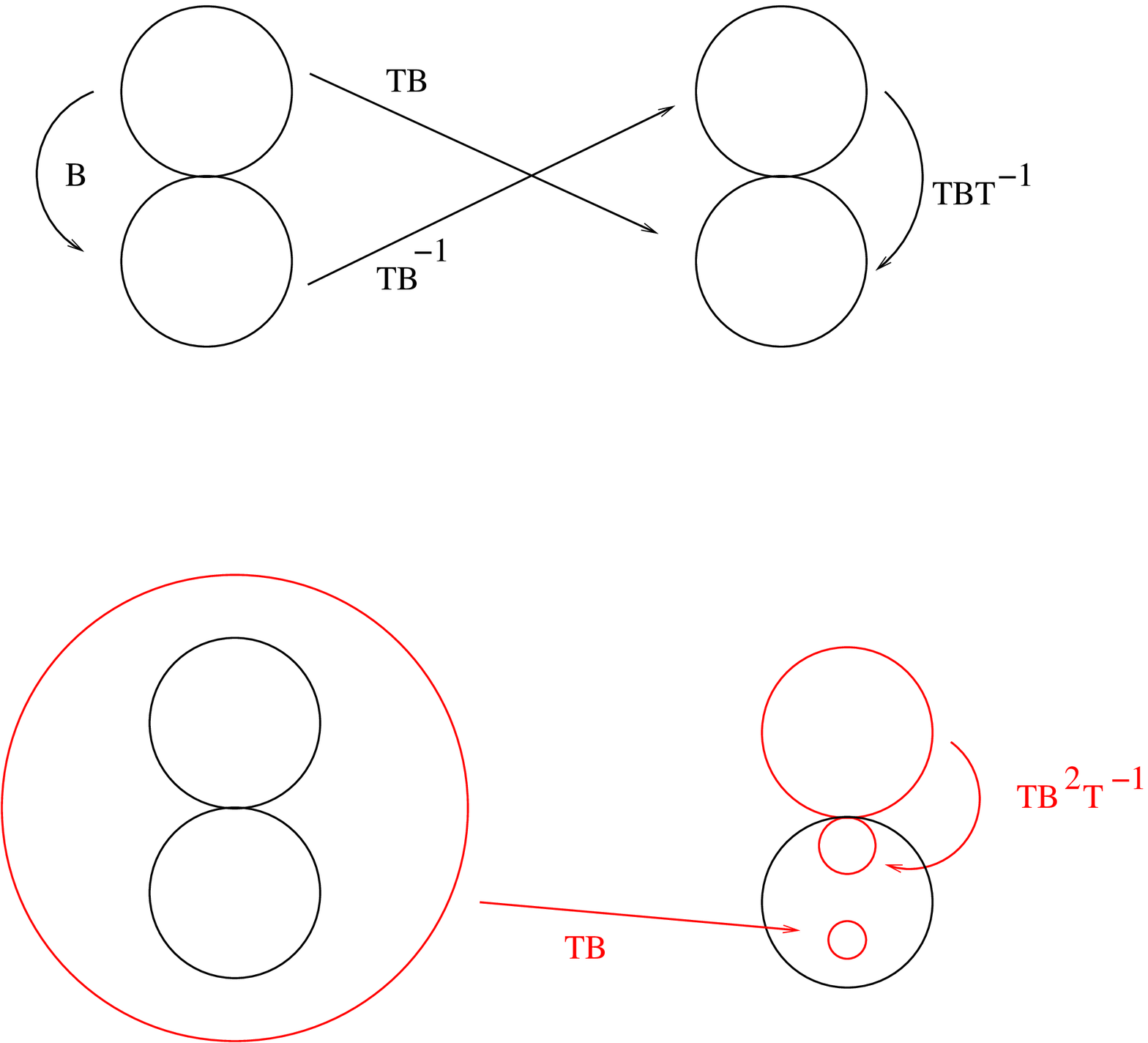}
\quad\quad
\includegraphics[width=0.3\linewidth]{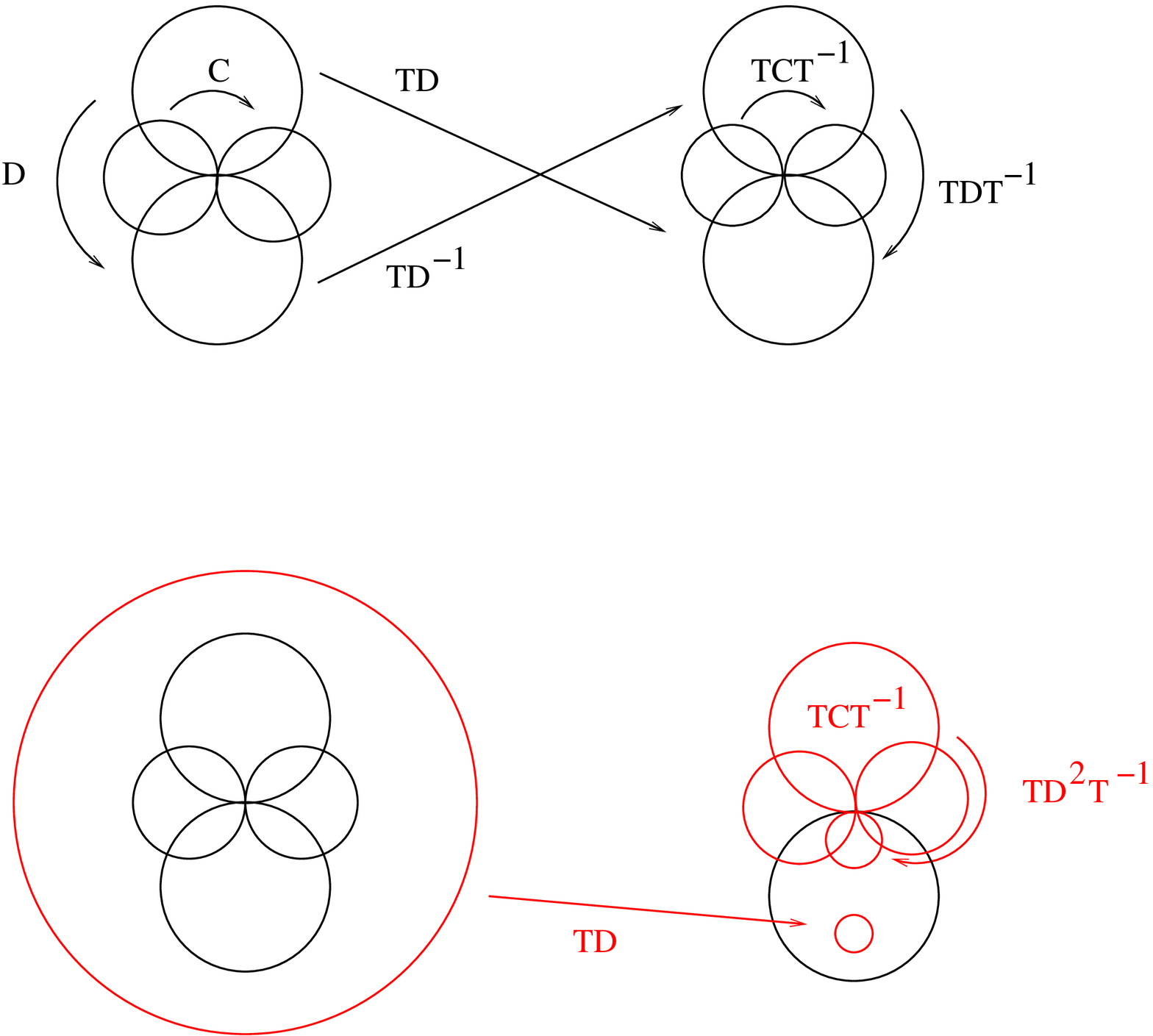}
\caption{Cases (1) and (2) in Lemma \ref{lemita}}
\label{Figura1}
\end{figure}



\begin{thebibliography}{99}
\bibitem{Chuckrow}
Chuckrow, V.
Subgroups and automorphisms of extended Schottky type groups.
{\it Trans. Amer. Math. Soc.} {\bf 150} (1970), 121--129.


\bibitem{Ker}
Ker\'ekj\'art\'o, B. 
{\it Vorlesungen \"uber Topologie I}. 
Mathematics: Theory \& Applications, Springer, Berl\'in, 1923.

\bibitem{Maskit:construction}
Maskit, B.
Construction of Kleinian groups.
{\it Proc. of the conf. on Complex Anal.}, Minneapolis, 1964, Springer-Verlag, 1965;
doi.org/10.1007/978-3-642-48016-4$\textunderscore$24 


\bibitem{Maskit:Comb}
Maskit, B.
On Klein's combination theorem III.
Advances in the Theory of Riemann Surfaces (Proc. Conf., Stony Brook, N.Y., 1969),
{\it  Ann. of Math. Studies} {\bf 66} (1971),
Princeton Univ. Press, 297-316.


\bibitem{Maskit:function2}
Maskit, B.
Decomposition of certain Kleinian groups.
{\it Acta Math.} {\bf 130} (1973), 243--263;
doi.org/10.1007/bf02392267 

\bibitem{Maskit:function3}
Maskit, B. 
On the classification of Kleinian Groups I. Koebe groups. 
{\it Acta Math.} {\bf 135} (1975), 249--270;
doi.org/10.1007/bf02392021 

\bibitem{Maskit:function4}
Maskit, B. 
On the classification of Kleinian Groups II. Signatures. 
{\it Acta Math.} {\bf 138} (1976), 17--42;
doi.org/10.1007/bf02392312 



\bibitem{Maskit:book}
Maskit, B. 
{\it Kleinian Groups}. 
GMW, Springer-Verlag, 1987.



\bibitem{Maskit:Comb4}
Maskit, B.
On Klein's combination theorem. IV.
{\it Trans. Amer. Math. Soc.} {\bf 336} (1993), 265-294.

\bibitem{Ian}
Richards, I. 
On the classification of noncompact surfaces. 
{\it Trans. Amer. Math. Soc.} {\bf 106} (1963), 259--269.




\bibitem{Shiga}
Shiga, H.
The quasiconformal equivalence of Riemann surfaces and the universal Schottky space.
{\it Conformal Geometry and  Dynamics} {\bf 23} (2019), 189--204.

\bibitem{Shiga2}
Shiga, H.
On the quasiconformal equivalence of dynamical Cantor sets.
{\it Journal d'Analyse Math\'ematique} {\bf 147} (2022), 1--28. 
\end{thebibliography}
\end{document}